\documentclass{article}
\usepackage{graphicx, amsfonts, amsmath, dsfont}
\usepackage{algpseudocode}
\usepackage{algorithm}
\usepackage{biblatex}
\usepackage{hyperref}
\hypersetup{colorlinks=true, urlcolor=blue, citecolor=blue}
\usepackage[margin=1.5in]{geometry}

\def\eps{\varepsilon}

\newcommand{\mF}{\mathcal{F}}

\newcommand{\mM}{\mathcal{M}}

\newcommand{\be}{\begin{equation}}
\newcommand{\ee}{\end{equation}}
\newcommand{\bes}{\begin{equation*}}
\newcommand{\ees}{\end{equation*}}

\newcommand{\bs}{\begin{split}}
\newcommand{\es}{\end{split}}
\newcommand{\bss}{\begin{split*}}
\newcommand{\ess}{\end{split*}}

\newcommand{\bmat}{\left[ \begin{matrix}}
\newcommand{\emat}{\end{matrix} \right]}
\newcommand{\bsmat}{\left[ \begin{smallmatrix}}
\newcommand{\esmat}{\end{smallmatrix} \right]}

\newcommand{\bml}{\begin{multline}}
\newcommand{\eml}{\end{multline}}
\newcommand{\bmls}{\begin{multline*}}
\newcommand{\emls}{\end{multline*}}

\usepackage[all]{xy}
\usepackage{amsthm}
\usepackage{amssymb, amsfonts}
\SelectTips{cm}{10}\UseTips

\theoremstyle{plain}
\newtheorem{thm}{Theorem}

\newtheorem{cor}[thm]{Corollary}
\newtheorem{lemma}[thm]{Lemma}

\newtheorem{fact}[thm]{Fact}

\theoremstyle{definition}
\newtheorem{definition}[thm]{Definition}

\newtheorem{note}[thm]{Note}

\numberwithin{thm}{section}
\numberwithin{equation}{section}

\title{Toroidal Embeddings of non-Intrinsically-Linked Graphs}
\author{Nathan Hall}
\date{November 2023}

\bibliography{references.bib}

\newcommand{\bbR}{\mathbb{R}}

\begin{document}

\maketitle

\setlength{\parindent}{0in}

\begin{abstract}
    If a graph $G$ can be embedded on the torus, and be embedded linklessly in $\bbR^3$, it's not known whether or not we can always find a linkless embedding of $G$ contained in the standard (unknotted) torus; We show that, for orders 9 and below, any graph which is both embeddable on the torus, and linklessly in $\bbR^3$, can be embedded linklessly in the standard torus.
\end{abstract}

\section{Background and Motivation}

A graph is said to be \textit{planar} if it can be embedded in $\bbR^2$, and \textit{toroidal} if it can be embedded into the torus. It's well-known that every planar graph is toroidal. If a graph $G$ can be embedded into $\bbR^3$ such that no pair of cycles in $G$ form a nontrivial link, we say that $G$ is \textit{nIL}, or non-intrinsically-linked. This is in contrast to graphs which are IL, or intrinsically linked, whose embeddings into $\bbR^3$ always contain a nontrivial link.

If we consider its standard (unknotted) embedding, we can think about the torus as a subset of $\bbR^3$. If a graph $G$ is both toroidal and nIL, we know there exists an embedding of $G$ into $\bbR^3$ such that it is entirely contained within the torus. We also know that there exists an embedding of $G$ into $\bbR^3$ which is not linked; but is there an embedding which is both unlinked \textit{and} restricted to the torus? Or does something about the geometry of the torus force a link to arise in certain graphs? We begin to answer this question by studying links in the torus.



\subsection{Links in the Torus}


We ask the reader to recall the \textit{linking number} of a link:

\begin{definition}
    Let $K$, $L$ be (disjoint) links. Consider a projection of $K \cup L$ into the $xy$-plane, and consider the points at which $K$ and $L$ cross. Assign a value of $+1$ or $-1$ to each crossing according to the following diagram: 

    \begin{figure}[h]
        \centering
        \includegraphics[scale=0.2]{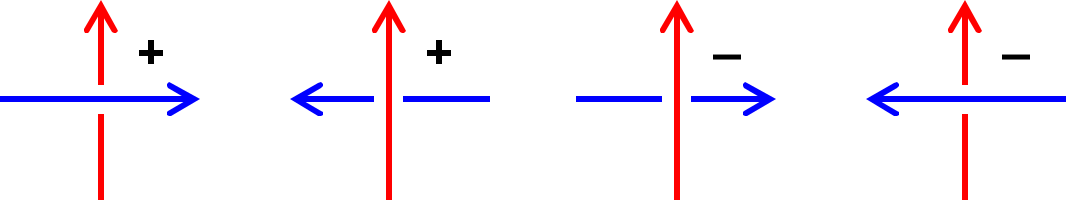}
        \label{crossdiagram1}
    \end{figure}

    Crossings which are assigned a value of $+1$ are called \textit{right-handed}, and crossings with a value of $-1$ are called \textit{left-handed}. The \textit{linking number} of $K\cup L$ is defined to be the sum of the values assigned to each crossing, divided by 2.
\end{definition}

This construction would not be half as useful as it is without the following theorem:

\begin{thm}
    The linking number is preserved under (ambient) isotopy; that is, if two links are equivalent, they have the same linking number.
\end{thm}

\begin{cor}
    If a link has nonzero linking number, it is nontrivial.
\end{cor}

The main relevance of the linking number for us is its connection to a similar construction which classifies knots in the torus: the \textit{slope} of a knot, defined below. The background showing the soundness of this definition is laid out, among other places, in \cite{Rolfsen_torus}.

\begin{definition}\label{def:slope}
    Let $K$ be a knot in the torus, letting $\ell$ denote the outer longitude and $m$ some meridian, both with some (arbitrary) orientation. Assume $K$ intersects $\ell$ and $m$ only transversally, and finitely many times. Using the convention that $K$ passes \textit{over} $\ell$ and $m$, assign to each crossing $+1$ if the crossing is right-handed and $-1$ if it is left-handed. Let $P$ be the algebraic sum of crossings of $K$ over $\ell$, and $Q$ the algebraic sum of crossings $K$ over $m$. If $P=Q=0$, then $K$ is said to be \textit{inessential}; otherwise we will say it has \textit{slope} $\frac{P}{Q}$.
\end{definition}

In the above definition, we may always assume that $P$ and $Q$ are coprime unless $K$ is inessential.
When dealing with a torus diagram, the slope of a knot is much easier to compute than the linking number, and as we will see the concepts are very much related.

\begin{thm}\label{thm:inessentialTrivial}
    Let $J,K$ be knots in the torus. If $J$ is inessential, then $J\cup K$ is trivial.
\end{thm}

\begin{proof}
    Since $J$ is inessential, it is null-homotopic according to \cite{Rolfsen_torus}; pick a point $p$ on the torus some (Euclidean) distance $\eps$ away from $K$ and isotop $J$ to a knot $L$ which is within $\frac{\eps}{2}$ of $p$. A 2-sphere of radius $\frac{\eps}{2}$ centered at $p$ intersects neither $L$ nor $K$, so $L\cup K$ is split, thus $J\cup K$ is trivial.
\end{proof}

\begin{thm}\cite{Rolfsen_torus}\label{thm:cycleSameSlope}
    Let $J,K$ be knots in the torus with well-defined slopes $j,k$ respectively. If $J$ and $K$ are disjoint, then $j=k$.
\end{thm}

\begin{definition}
    Let $m,n$ be integers with $\gcd(m,n)=d$. The torus link $T(m,n)$ is defined to be $d$-many disjoint copies of a knot of slope $\frac{m}{n}$.
\end{definition}

\begin{thm}[\cite{Bush_French_Smith_2014}]\label{thm:torusLKnum}
    The linking number of the torus link $T(m,n)$ is $\frac{mn}{2}(1-\frac{1}{\gcd(m,n)})$.
\end{thm}

As pointed out in a MathStackExchange post\cite{Bargabbiati_LowranceAdam}, the formula for the linking number of a torus link stated in \cite{Bush_French_Smith_2014} has a typo; the formula above is the correct one, which agrees with the authors' computations.

\begin{thm}\label{thm:linkslopepair}
    Let $L_1$ and $L_2$ be disjoint knots in the torus. Then $L_1\cup L_2$ forms a link if and only if both knots have well-defined slope $\frac{n}{m}$, where $m,n\neq 0$.
\end{thm}

\begin{proof}
    First suppose that $L_1$ and $L_2$ have well-defined slope $\frac{n}{m}$ with $m,n\neq 0$. (We will also assume $m,n$ are coprime.) Then $L_1\cup L_2$ is (ambient) isotopic to $T(2m,2n)$, which has a linking number of $\frac{(2m)(2n)}{2}\big(1-\frac{1}{\gcd(2m,2n)}\big)$ by \ref{thm:torusLKnum}. We have $\gcd(2m,2n)=2$, meaning the linking number is $mn\neq 0$, and thus $L_1\cup L_2$ is nontrivial. 

    Now suppose the negation of the above. If WLOG $L_1$ is inessential, then $L_1\cup L_2$ is trivial by \ref{thm:inessentialTrivial}. Otherwise, $L_1$ and $L_2$ must have the same slope by \ref{thm:cycleSameSlope}. If $L_1$ and $L_2$ have slope $0$, then $L_1 \cup L_2$ is isotopic to the first link in Figure \ref{fig:slopeCounterexamples}, which is trivial by inspection. Similarly, if $L_1, L_2$ have slope $\infty$, then $L_1 \cup L_2$ is isotopic to the second link in Figure \ref{fig:slopeCounterexamples}, also trivial. This fully proves the result.
\end{proof}

\begin{figure}[ht]
    \centering
    \includegraphics[width=0.7\linewidth]{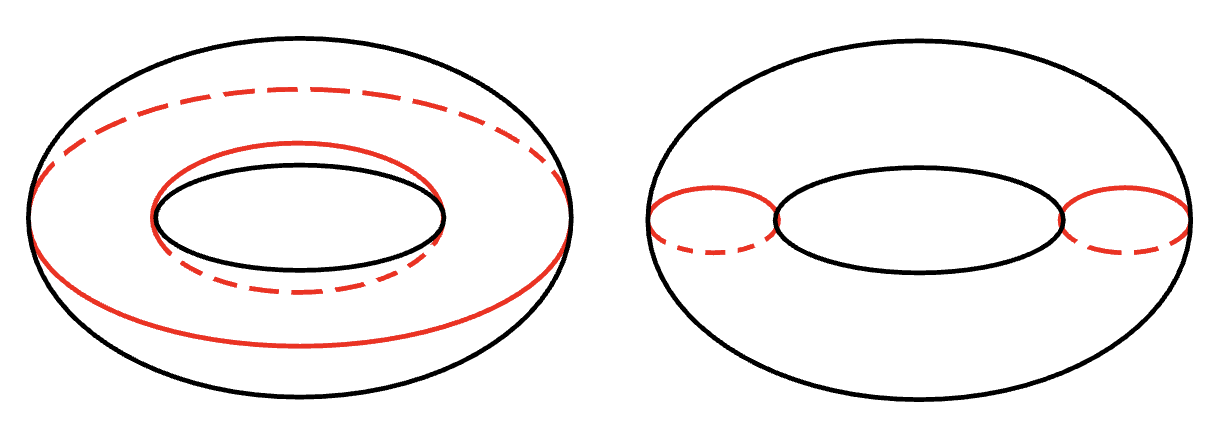}
    \caption{Two trivial links which arise from slopes 0 and $\infty$}
    \label{fig:slopeCounterexamples}
\end{figure}

The implications of this theorem for graphs in the torus will be discussed in more detail in Section 2.3; for now we turn to another essential part of our investigation.

\subsection{Minor-Closed Families}

A set or family of graphs $F$ is said to be \textit{minor-closed} if, for any $G\in F$, any minor of $G$ is a member of $F$. A graph $M$ is said to be a \textit{forbidden minor} for $F$ if it has the following property: if $M$ is a minor of $G$, then $G$ is not a member of $F$. The famous Robertson-Seymour Theorem states that the set of forbidden minors for any minor-closed family must be finite.

It's well-known that the family of nIL graphs is minor-closed, with the Petersen family of graphs forming the complete set of forbidden minors. It's also known that the set of toroidal graphs is minor-closed, but the complete set of forbidden minors is not known; more on this later.

We wish to know whether graphs which are both toroidal and nIL are embeddable linklessly on the torus, and it so happens that we can analyze these graphs as members of minor-closed families. From here on, if a graph is toroidal and nIL, we will say it is ``TN," and if a graph has a linkless embedding on the torus, we will say it is \textit{linklessly toroidally embeddable}, or ``LTE".

The reader will note that the family of LTE graphs is a subset of the family of TN graphs, as any graph which is LTE is both toroidal and nIL. The main question of this paper, as posed on the first page, is equivalent to asking whether the inclusion goes the opposite direction as well. 

It's easy to see that a the family of TN graphs is minor-closed, given that the families of toroidal graphs and nIL graphs are independently minor-closed. This, however, does not imply that the family of LTE graphs is minor-closed. This fact is not especially difficult to show; however, the proof of closure under edge contraction is annoyingly involved, and not necessary for our purposes, so we omit it here and only include closure under taking subgraphs.

\begin{thm}\label{thm:LTESubgraphClosed}
    Any subgraph of an LTE graph is LTE.
\end{thm}

\begin{proof}
    Let $G$ be an LTE graph with toroidal embedding $E$, and let $G'\subseteq G$. Then restricting the embedding $E$ to $G'$ results in an embedding whose image is a subset of $E$; since $E$ is linkless and toroidal, this new embedding is also linkless and toroidal, thus $G'$ is LTE.
\end{proof}

We will also make frequent use of maximal elements of a minor-closed family; by a maximal element, we mean a graph $G$ where for any $e$ in the edge set of its complement, adding $e$ to $G$ results in a graph which is not in the family.

\begin{fact}\label{lem:MaxlSubgraph}
    Let $\mF$ be a minor-closed family, and let $G\in \mF$ be a graph of order $n$. Then $G$ is a subgraph of a maximal member of $\mF$ of order $n$.
\end{fact}

The proof of this fact is trivial; just add edges to $G$ until either there are no edges left to add, or adding any edge results in a graph which is not in $\mF$.

If a graph is maximally nIL, we say it is ``maxnIL." If a graph is maximally TN, we say it is ``MTN."

\subsubsection{MaxnIL Graphs and Toroidal Obstructions}

The complete sets of maxnIL graphs of orders up to 11 are known through the work of \cite{naimi2022complement}. We concern ourselves only with maxnIL graphs of orders 6 through 9 for now; there is just one maxnIL graph of order 6 ($K_6-e$), two maxnIL graphs of order 7, six of order 8, and twenty of order 9.

It was mentioned earlier that the complete set of forbidden minors for toroidal graphs is not known. The work of Myrvold and Woodcock \cite{Myrvold_Woodcock_2018} is the most expansive work we know of attempting to classify this set; they have determined every toroidal forbidden minor of order 12 and below. The authors refer to these graphs as toroidal minor-order obstructions; we will refer to them as toroidal obstructions or simply as obstructions. This work leads to the following (critical) theorem:

\begin{thm}[corollary of \cite{Myrvold_Woodcock_2018}]\label{thm:ToroidalObstructions}
    Let $B_n$ be the set of obstructions of order $n$ given in \cite{Myrvold_Woodcock_2018}. A graph of order $n$, where $n\leq 12$, is toroidal if and only if it does not contain any graph in $B_k$ as a minor for $k\leq n$.
\end{thm}

The smallest order of any obstruction is 8, of which there are exactly three: the first is isomorphic to $K_8 - K_3$, by which we mean $K_8$ excluding the edges of $K_3$. The other two are isomorphic to $K_8-(2K_2 \cup P_3)$ and $K_8 - K_{2,3}$; these three graphs have size 25, 24, and 22 respectively. 

All maxnIL graphs as well as toroidal obstructions of order 9 and below, all taken directly from the source material in \cite{naimi2022complement} and \cite{Myrvold_Woodcock_2018}, can be found in the attached Mathematica file \verb|ImportantGraphs.nb|.

\section{Method}

\subsection{Basis}

Our method relies on the following result:

\begin{thm}
    If every MTN graph of order $n$ is LTE, then every TN graph of order $n$ is LTE.
\end{thm}
\begin{proof}
    If every MTN graph of order $n$ is LTE, then every subgraph of an MTN graph is LTE by \ref{thm:LTESubgraphClosed}.
    By \ref{lem:MaxlSubgraph}, every TN graph of order $n$ is a subgraph of some MTN graph of order $n$, giving the result.
\end{proof}

According to this theorem, finding linkless, toroidal embeddings of every MTN graph of order $n$ is sufficient to prove that every TN graph of order $n$ is LTE. This approach allows us to choose orders as large as we'd like, as long as we know the following: firstly, what is the set of MTN graphs of order $n$? Once we know which graphs are in this set, finding embeddings can be done by hand. Secondly, once we have an embedding of a graph, how can we determine whether the embedding is linkless? These two questions are answered in the next two sections.

\subsection{Finding MTN Graphs}

\subsubsection{Orders 8 and Below}

We begin finding MTN graphs by leveraging our existing knowledge of maxnIL graphs. Note that every toroidal, maxnIL graph is MTN, since the graph itself is TN and adding any edge results in a graph which is IL (i.e, not TN). For each order, then, the work lies in finding those MTN graphs which are not maxnIL, or proving that none exist. For the latter option, the following lemma is useful.

\begin{lemma}\label{lem:MTNmaxnIL}
    Suppose that every maxnIL graph of order $n$ is toroidal, and let $G$ be a graph of order $n$. Then $G$ is MTN if and only of $G$ is maxnIL.
\end{lemma}
\begin{proof}
    The reverse direction is proven by the note above, so we need only show the forward direction. Let $G$ be an MTN graph of order $n$. Then $G$ is nIL, and so must be a subgraph of some maxnIL graph $M$ of order $n$ by \ref{lem:MaxlSubgraph}. Note that $M$ is TN by assumption. Now suppose toward contradiction that $G$ is a proper subgraph of $M$. Then there exists an edge $e$ such that $G+e\subseteq M$, implying $G+e$ is TN. This contradicts that $G$ is MTN; therefore $G=M$, i.e. $G$ is maxnIL.
\end{proof}

The above lemma equivalently says that the set of maxnIL graphs of order $n$ is also the complete set of MTN graphs for that order, as long as all of the maxnIL graphs are toroidal. We can use this result to immediately prove our main result for a few small orders.

We start at order 6, since no graph of order 5 or below can be IL. There is one maxnIL graph of order 6 ($K_6-e$) and two maxnIL graphs of degree 7. Recall that the smallest toroidal obstruction is of order 8; as such we can immediately say that all three of these graphs are toroidal, and by \ref{lem:MTNmaxnIL}, they are MTN as well.
Furthermore, note that since $K_5$ is a subgraph of $K_6-e$, every graph of order 5 and below is also a subgraph. Handily, this means that proving $K_6-e$ is LTE implies that the result holds for every graph of order 5 and below.

For order 8, we can also use \ref{lem:MTNmaxnIL} to our advantage. As mentioned, the smallest toroidal obstruction is $K_8-K_{2,3}$, which has 8 vertices and 22 edges. There are 6 maxnIL graphs of order 8; one has size 21 and the rest have size 22, as can be seen in \verb|ImportantGraphs.nb|. Thus the only way any of these graphs can be non-toroidal is if they are isomorphic to $K_8-K_{2,3}$. However, this cannot be the case, as $K_8-K_{2,3}$ is IL; it has $K_6$ as a subgraph, as can be seen in Figure \ref{fig:DH3IL}. Thus, the six maxnIL graphs form the complete set of MTN graphs for order 8.

\begin{figure}[ht]
    \centering
    \includegraphics[width=0.5\linewidth]{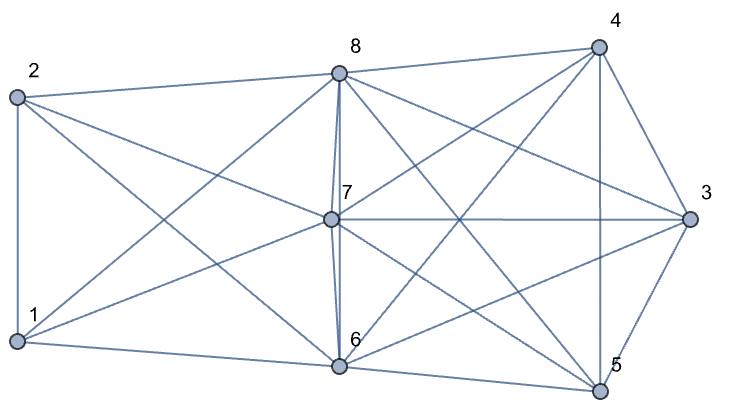}
    \centering
    \caption{$K_8-K_{2,3}$. Partite sets for the removed edges are \{1,2\} and \{3,4,5\} respectively; deleting vertices 1 and 2 yields $K_6$.}
    \label{fig:DH3IL}
\end{figure}

\subsubsection{Order 9}

Order 9 is the smallest order for which some of the maxnIL graphs are non-toroidal, so Theorem \ref{lem:MTNmaxnIL} is no longer of use. We know that every toroidal maxnIL graph is MTN, but in order to find the complete set of MTN graphs we must also find those MTN graphs which are not maxnIL. Throughout this section, we will use the label $\mM$ for the set of all order-9 MTN graphs which are not maxnIL. We will determine $\mM$ by using a recursive search algorithm, but we begin by narrowing the search space with lemmata and smaller pieces of computation.

\begin{note}
    Recall Theorem \ref{thm:ToroidalObstructions}: a graph is non-toroidal if and only if it contains an obstruction as a minor. If a graph $G$ of order 9 is toroidal, the only obstructions it could contain are those of order 8 and 9; but any order-9 obstruction it contains would have to be strictly as a subgraph.
\end{note}

We use this fact in the following two lemmata, whose proofs can be found in the accompanying Mathematica file \verb|Sec_2.2.2|.

\begin{lemma}\label{lem:mnILnonTors}
    Of the twenty maxnIL graphs of order 9, only four are non-toroidal.
    \begin{figure}[ht]
        \centering
        \includegraphics[width=0.22\linewidth]{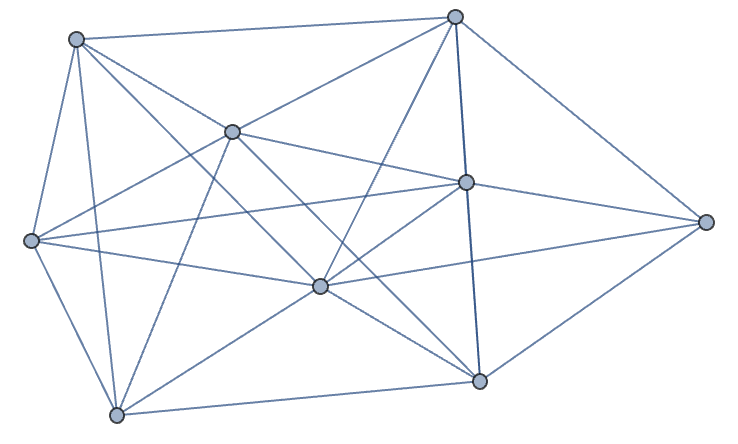}
        \includegraphics[width=0.22\linewidth]{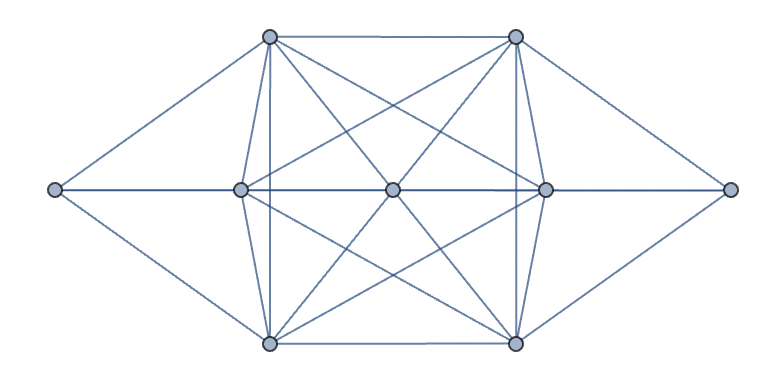}
        \includegraphics[width=0.22\linewidth]{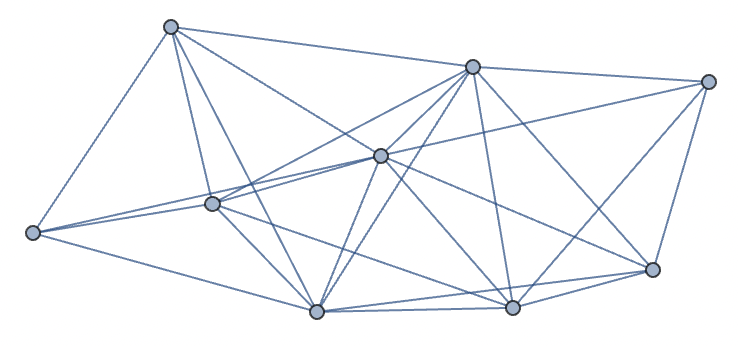}
        \includegraphics[width=0.22\linewidth]{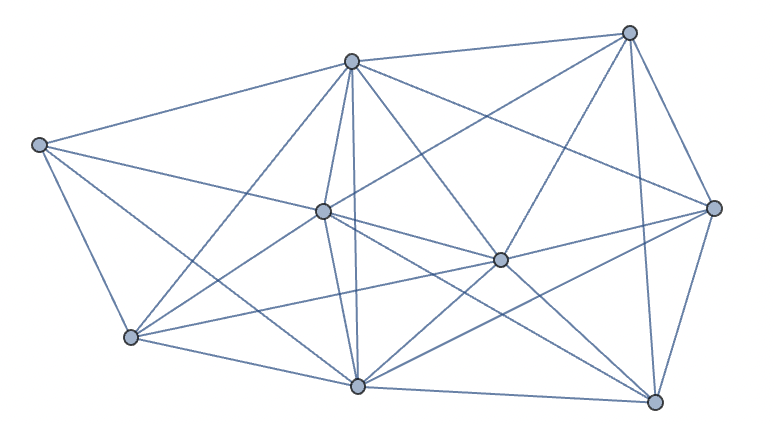}
        \caption{The four non-toroidal maxnIL graphs of order 9.}
        \label{fig:mnILNT}
    \end{figure}
\end{lemma}

It will be useful to know exactly which obstructions are contained in these four graphs. From here on, let $S$ denote the set of all obstructions contained in some non-toroidal maxnIL graph.

\begin{lemma}\label{lem:mnILnonTorObsts}
    The set $S$ consists of the following five graphs:
    \begin{figure}[ht]
        \centering
        \includegraphics[width=0.3\linewidth]{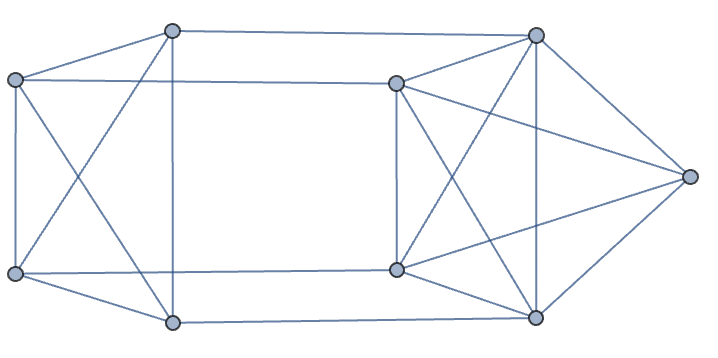}
        \includegraphics[width=0.3\linewidth]{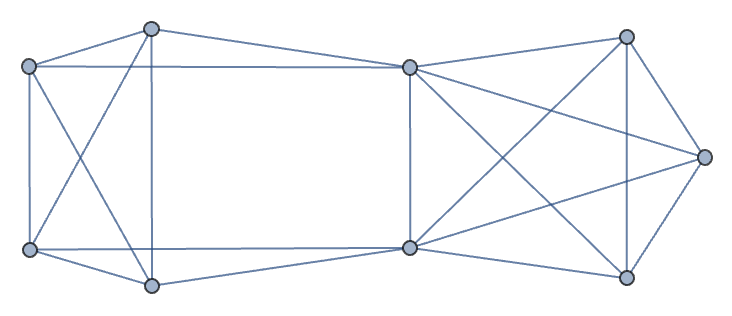}
        \includegraphics[width=0.3\linewidth]{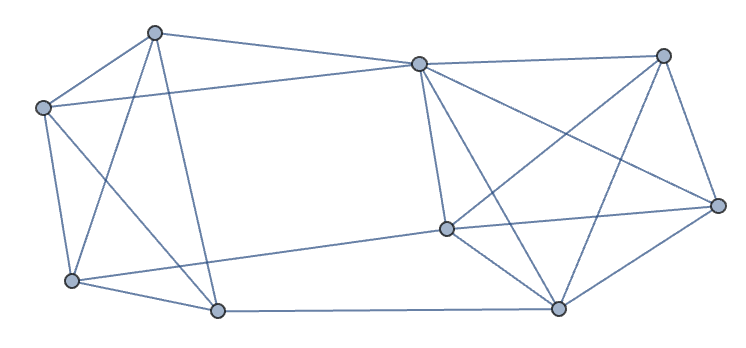}
        \includegraphics[width=0.3\linewidth]{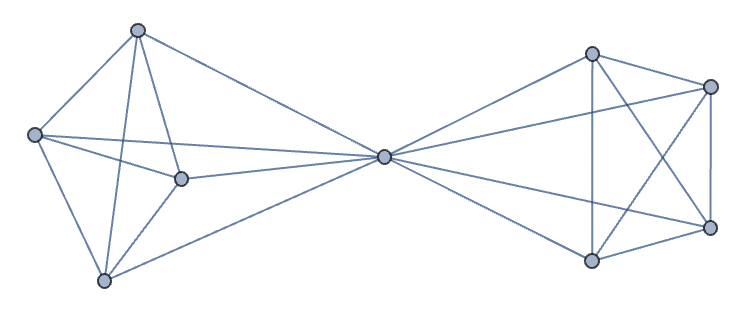}
        \includegraphics[width=0.3\linewidth]{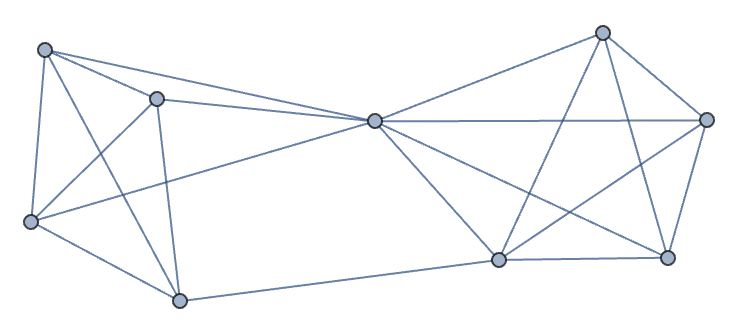}
        \label{fig:torObst}
    \end{figure}
\end{lemma}

\begin{fact}\label{fac:SHaveSize20}
    All members of $S$ have size 20.
\end{fact}

The reader may verify the above fact by counting the edges in each graph in the above figure; however for convenience we also include a cell in \verb|Sec_2.2.2| to accomplish this. 

We also include a consequence of \ref{lem:mnILnonTors}.

\begin{cor}\label{cor:nonTorIFFsize20subg}
     If a nIL graph $G$ of order 9 is non-toroidal, then some member of $S$ must be a subgraph of $G$.
\end{cor}

\begin{lemma}\label{lem:subgraphOfNandNotM}
    Let $M$ denote the set of toroidal, maxnIL graphs of order 9, $N$ denote the set of non-toroidal maxnIL graphs of order 9, and let $G$ be an $MTN$ graph of order 9. The following are equivalent:
    \begin{itemize}
        \item[i.] $G$ is not maxnIL; i.e. $G\in \mM$.
        \item[ii.] $G$ is a proper subgraph of some member of $N$
        \item[iii.] $G$ is not a subgraph of any member of $M$.
    \end{itemize}
\end{lemma}

\begin{proof}
    ($i\rightarrow iii)$ Suppose $G$ is not maxnIL, and suppose toward contradiction that $G\subseteq H$ for some $H\in M$. Since $H$ is maxnIL and $G$ is not, $G\neq H$, so $G\subset H$; i.e. there exists some edge $e$ such that $G+e\subseteq H$. Because $G$ is $MTN$, the graph $G+e$ is either IL or non-toroidal, a contradiction since $H$ is TN. So $G\not \subseteq H$ for any $H\in M$.

    ($iii\rightarrow ii$) Suppose $G$ is not a subgraph of any member of $H$. By \ref{lem:MaxlSubgraph}, $G$ must be a subgraph of some maxnIL graph of order 9 because $G$ is nIL. Thus it must be a subgraph of some member of $N$, and because $G$ is toroidal, it must be a proper subgraph.

    ($ii \rightarrow i$) Suppose $G$ is a proper subgraph of some $H\in N$; i.e. there exists some edge $e$ such that $G+e\subseteq H$. Then $G+e$ is nIL, so $G$ is not maxnIL.
\end{proof}

\begin{lemma}\label{lem:MTNConnected}
    No MTN graph of order 9 is disconnected.
\end{lemma}

\begin{proof}
    Let $G$ be a disconnected, TN graph of order 9. Suppose $G=A\sqcup B$ where $A\subseteq K_n$ and $B\subseteq K_{9-n}$ given WLOG $1\leq n \leq 4$; note that $A$ and $B$ need not be connected themselves. Since the family of TN graphs is minor-closed, $A$ and $B$ are both TN; moreover $A$ is planar since it has fewer than 5 vertices. We will show that we can always pick a vertex $a\in A$, $b\in B$ such that $G+\{a,b\}$ is TN.

    Let $E_A$ be a planar embedding of $A$ and pick a vertex $a\in A$ so that $a$ is not enclosed by a cycle in the plane. Let $E_B$ be a toroidal embedding of $B$ and $b\in V(B)$, and let $e=\{a,b\}$. In $E_B$, the vertex $b$ will be incident to some face, i.e. a region of $T^2-E_B$ which is homeomorphic to $\bbR^2$. We can now construct a toroidal embedding of $G+e$ by embedding $E_A$ into this region and connecting $a$ to $b$ with the edge $e$, as in Figure \ref{fig:GplusEisTor}. Thus, $G+e$ is toroidal. 

    \begin{figure}[ht]
        \centering
        \includegraphics[width=0.5\linewidth]{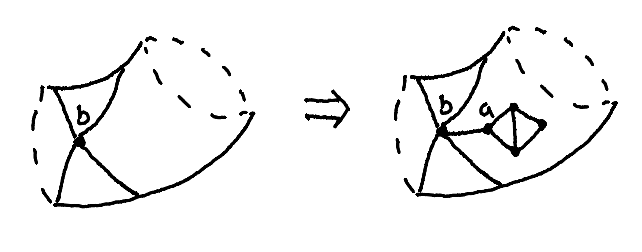}
        \caption{A toroidal embedding of $G+e$.}
        \label{fig:GplusEisTor}
    \end{figure}
    
    Since $A$ and $B$ are TN, they have linkless embeddings in $\bbR^3$, which we will now call $E_A$ and $E_B$. Thus we can create an embedding of $G+e$ by connecting $E_A$ and $E_B$ with, as in Figure \ref{fig:GplusEisNil}. If $G+e$ were IL, then this embedding would contain a link; however $e$ is not part of a cycle since it is a cut-edge for some component of $G+e$. This implies that $A$ or $B$ contains a link, contradicting that $E_A$ and $E_B$ are linkless. Thus $G+e$ is nIL.

    \begin{figure}[ht]
        \centering
        \includegraphics[width=0.4\linewidth]{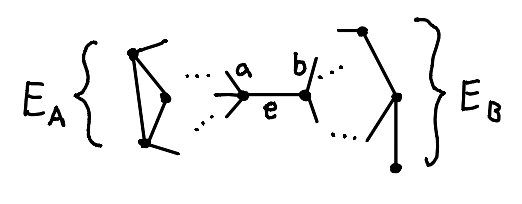}
        \caption{A linkless embedding of $G+e$.}
        \label{fig:GplusEisNil}
    \end{figure}
\end{proof}

\begin{lemma}\label{lem:LB1}
    No member of $\mM$ has size $\leq 18$. 
\end{lemma}

\begin{proof}
    Let $G$ be an MTN graph of order 9, and suppose that $G$ has size $\leq 18$. Now suppose toward contradiction that $G+e$ is nIL for some edge $e$. Then $G+e$ must be non-toroidal, as $G$ is MTN, and by Corollary \ref{cor:nonTorIFFsize20subg} must contain a member of $S$ as a subgraph. But $G+e$ has size $\leq 19$, whereas all members of $S$ have size 20, yielding a contradiction. So $G+e$ must be IL, implying that $G$ is maxnIL, and in particular that $G\not\in \mM$.
\end{proof}

Increasing lower bound on graph size drastically reduces the number of candidate graphs for $\mM$ which we will have to search through. In the following lemma we again use computation to restrict the search space again; this will be the final restriction.

\begin{lemma}\label{lem:LB2}
    No member of $\mM$ has size 19.
\end{lemma}

\begin{proof}
    Let $G\in \mM$ and suppose that $G$ has order 19. Since $G\in \mM$, there must be at least one edge $e$ such that $G+e$ is nIL, otherwise $G$ would be maxnIL. By assumption we have that $G+e$ is non-toroidal and therefore that some member of $S$ is a subgraph of $G+e$; but since it has size 20, this gives $G+e\in S$. Thus $G=H-e$ for some $H\in S$ and some edge $e\in H$. 

    In the attached Mathematica file \verb|Sec_2.2.2|, the function \verb|Lem2Point11| searches through all possible graphs of the form $H-e$ and verifies that there is always an edge $e'$ such that $H-e+e'$ is both toroidal and nIL. Since those were the only candidates for graphs of size 19 in $\mM$, this forms a proof by exhaustion that there are no such graphs in $\mM$.
\end{proof}

This handful of lemmata allows us now to put forth a new tool in our search for MTN graphs. This tool is a recursive algorithm, which we will call \verb|MTNSearch|, displayed as Algorithm \ref{alg:MTNSearch} below, which takes as input a graph $G$ and outputs a specific list of subgraphs of $G$. The algorithm will allow us to narrow our search space significantly, as shown in Theorem \ref{thm:MTNSearchWorks}.

Algorithm \ref{alg:MTNSearch} also requires us to be able to tell algorithmically when a graph is nIL; for this we use an algorithm put forth in \cite{naimi2020intrinsicallyknottedlinkedgraphs}.

\begin{algorithm}[ht]\caption{ 
 MTNSearch(G)}\label{alg:MTNSearch}
\begin{algorithmic}
    \Require{G is a nIL graph of order 9}
    \If{Size(G)$>$19 \textbf{and} G is connected \textbf{and} G is not a subgraph of a toroidal maxnIL graph}
    \If{G is not toroidal}
    \State \Return{$\bigcup_{e\in E(G)}$MTNSearch($G-e$)}
    \Else \State \Return{\{G\}}
    \EndIf
    \Else \State \Return \{\}
    \EndIf
\end{algorithmic}
\end{algorithm}

\begin{thm}\label{thm:MTNSearchWorks}
    If $H\in \mM$ and $H\subseteq G$, then $H\in MTNSearch(G)$.
\end{thm}
\begin{proof} Let $H\in \mM$. We proceed by induction. 

For our base case, we take $G=H$. Note that $H$ has more than 19 edges by \ref{lem:LB1} and \ref{lem:LB2}, is connected by \ref{lem:MTNConnected}, and is not a subgraph of any toroidal maxnIL graph by \ref{lem:subgraphOfNandNotM}. Moreover it is toroidal (and nIL) by assumption; thus MTNSearch($H$) returns $\{H\}$.

For our inductive case, assume that the result holds for $G-e$ for each $e\in G$, and suppose that $H\subset G$. Then $G$ has more than 19 edges since $H$ does, and $G$ is connected since $H$ is connected. Moreover, $G$ is not a subgraph of any toroidal maxnIL graph, as this would imply that $H$ is as well, impossible by lemma. Finally, since $H\subset G$, we have $H+e\subseteq G$ for some $e\not\in H$. $G$ is nIL by assumption, implying that $H+e$ is nIL and thus non-toroidal since $H$ is MTN. Thus $G$ is also non-toroidal, meaning MTNSearch($G$) returns the set $\bigcup_{e\in E(G)}$MTNSearch($G-e$). Because $H \subset G$, we have that $H\subseteq G-e$ for some edge $e\in G$; thus MTNSearch($G-e$) contains $H$ by inductive hypothesis, and so $H\in$ MTNSearch($G$).
\end{proof}

\begin{note}\label{note:unionMTNSearch}
    The corollary to this is clear: since every graph in $\mM$ must be a subgraph of some non-toroidal maxnIL graph, then running \verb|MTNSearch| on each of the non-toroidal maxnIL graphs and taking the union yields a set of which $\mM$ is a subset. From there, it will be easy to brute-force check whether each of the graphs in this set is MTN, and we will be able to find $\mM$.
\end{note}

We now turn to the second goal, which is a way to prove that an embedding of an MTN graph is linkless.

\subsection{Proving Linklessness}

In Section 1.1, we defined the slope of a knot in the torus in terms of its crossings over the outer longitude and some meridian. We will now connect this idea to the slope of a cycle in an embedding. 

First, note that we will be dealing primarily with torus diagrams, in which the top and bottom sides of a square are glued together, then the left and right sides. By convention, the top-bottom boundary runs along the outer meridian of the torus and the left-right boundary forms a meridian. Thus the slope of a cycle in a torus diagram can be determined by taking the algebraic sums of the crossings of these boundaries. Since the orientation of the longitude and meridian are arbitrary, we will choose the orientation so that a crossing off the top of the diagram back around to the bottom is right-handed (positive) and the opposite direction is left-handed (negative). A crossing off the right boundary of the diagram around to the left is also right-handed, and left-to-right is left-handed.

We will assume that an edge in a cycle will cross a boundary at most once, so only edges which cross the boundary contribute to the slope. Thus, for a toroidal embedding $E$, an algorithm knowing which edges cross each boundary, and on which side of the boundary each endpoint lies, would be able to find the slope of any cycle in $E$ given to it.

In Algorithm \ref{alg:FindLinks}, UpList is the list of edges which cross the upper boundary. Elements of UpList are ordered tuples $(u,v)$, with $u$ situated below the boundary and $v$ situated above the boundary, and correspondingly for RightList. We will assume the vertices of a graph of order $n$ may be represented as the integers 1 through $n$, and that a cycle is represented as an ordered list of vertices (integers).

\begin{algorithm}[ht]\caption{FindLinksInTorusEmbedding(G,UpList, RightList)}\label{alg:FindLinks}
\begin{algorithmic}
    \Require{$G$ a toroidal graph with vertices $\{1...n\}$, UpList and RightList as above}
    \State Let LinkList = \{\}
    \State Let CycSlopeList = \{\};
    \State Let $M$ be an $n\times n$ matrix with each entry being the ordered pair $(0,0)$;
    \For{each edge $(j,k)\in$ UpEdges}
    \State $M_{j,k}$ += (1,0), $M_{k,j}$ += ($-1$,0);
    \EndFor
    \For{each edge $(j,k)\in$ RightEdges}
    \State $M_{j,k}$ += (0,1), $M_{k,j}$ += (0,$-1$)
    \EndFor
    \For{each cycle $C={v_1,...,v_k}$ in $G$ where $3\leq k \leq n-3$}
    \State Let $(P,Q)=[\sum_{i=1}^{k-1} M_{v_i,v_{i+1}}]+M_{v_k,v_1}$;
    \Comment{With addition taken to be component-wise}
    \If{$P\neq 0$ and $Q\neq 0$}
    \State Append $(\frac{P}{Q},C)$ to CycSlopeList;
    \EndIf
    \EndFor
    \State For each distinct slope $k$ in CycSlopeList, create a list $S_k$ of every cycle with slope $k$;
    \For{each $S_k$}
    \For{each pair of cycles $C_1,C_2$ in $S_k$}
    \If{$C_1$ and $C_2$ are disjoint}
    \State Append $(C_1,C_2)$ to LinkList;
    \EndIf
    \EndFor
    \EndFor \\
    \Return LinkList
\end{algorithmic}
\end{algorithm}

\begin{thm}
    Let $E$ be a toroidal embedding of a graph $G$. Then $E$ is linkless if and only if the list returned by Algorithm \ref{alg:FindLinks} is empty.
\end{thm}

The proof of this theorem follows from the above argument combined with Theorem \ref{thm:linkslopepair}.

\section{Results}

\subsection{MTN Graphs of Order 9}

The computation specified in Note \ref{note:unionMTNSearch} is carried out at the end of the attached Mathematica notebook \verb|Sec_2.2.2|; on the author's local machine, computation took roughly two and a half hours. The resulting set $\mM$, the set of all non-maxnIL MTN graphs of order 9, is a set of eleven graphs, whose renderings are given in Figure \ref{fig:ScriptM}. These eleven graphs, combined with the sixteen toroidal maxnIL graphs, form the complete set of MTN graphs of order 9. 

\begin{figure}[ht]
    \centering
    \includegraphics[width=\linewidth]{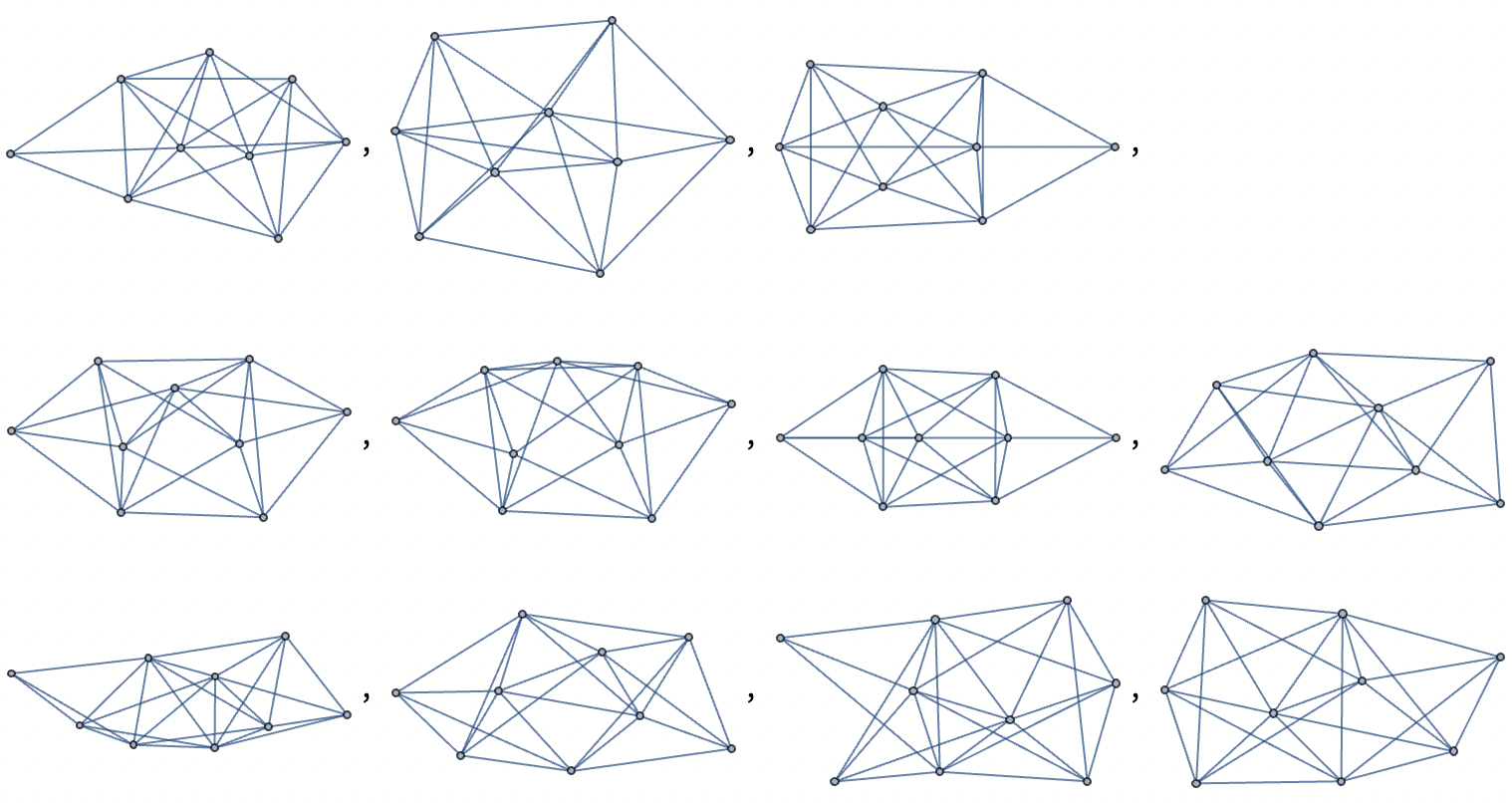}
    \caption{All MTN graphs of order 9 which are not maxnIL.}
    \label{fig:ScriptM}
\end{figure}

\subsection{Embeddings of MTN Graphs}

The embeddings of all twenty-seven MTN graphs of order 9, as well as the six MTN graphs of order 8, the two of order 7, and the one of order 6, were found by hand. The use of the Mathematica tool ``Embeddings of Graphs in a Torus and in a Moebius Strip" \cite{Rytin_2011} developed by Maxim Rytin, and published through the Wolfram Demonstrations Project, was indispensable in this process.

The embeddings themselves are included in the included document \verb|MTNEmbeddings.pdf| and the reader may verify the linklessness of each of them in the Mathematica notebook \verb|MTN_Embeddings.nb|, where the proper inputs to Algorithm \ref{alg:FindLinks} are shown.

This completes the proof that all TN graphs of order 9 and below are LTE.

\subsection{Future Directions}

Based on these results, we conjecture that the result proven for small orders holds for every order; i.e. that every toroidal, nIL graph has a linkless embedding in the torus. Attempting to prove this piecewise for higher orders using the same general approach we have used in this paper is unlikely to be worth the tradeoff, as the number of MTN graphs grows exponentially with order. An approach that may be feasible in general is showing that the set of forbidden minors for TN graphs and for LTE graphs is identical; difficult since the full set of toroidal forbidden minors is not fully known, but not impossible. 

\section{Acknowledgements}

This paper would not have been possible in any sense of the word without the guidance and expertise of Prof. Ramin Naimi. Thank you for giving me this opportunity, and for the patience and contributions that have allowed it to become what it is today.

\vspace{0.5cm}

\hrule

\vspace{0.5cm}

All code for this project can be found at \href{https://github.com/hall-nate/TorNilEmb}{https://github.com/hall-nate/TorNilEmb}.

\printbibliography

\end{document}